\documentclass{amsart}
\usepackage[all]{xy}
\usepackage{amsmath,amssymb,amsthm,mathrsfs,bm}

\DeclareMathOperator*{\restprod}%
 {\mathchoice{\ooalign{\ensuremath{\displaystyle\prod}\crcr\ensuremath{\displaystyle\coprod}}}%
             {\ooalign{\ensuremath{\textstyle\prod}\crcr\ensuremath{\textstyle\coprod}}}%
             {\ooalign{\ensuremath{\scriptstyle\prod}\crcr\ensuremath{\scriptstyle\coprod}}}%
             {\ooalign{\ensuremath{\scriptscriptstyle\prod}\crcr\ensuremath{\scriptscriptstyle\coprod}}}%
 }

\theoremstyle{definition}
\newtheorem{dfn}{Definition}[section]

\newtheorem{rmk}[dfn]{Remark}

\theoremstyle{plain}
\newtheorem{prop}[dfn]{Proposition}
\newtheorem{lem}[dfn]{Lemma}
\newtheorem{thm}[dfn]{Theorem}

\def \O{\mathcal{O}}
\def \Kab{K^{\mathrm{ab}}}
\def \GKab{G_K^{\mathrm{ab}}}
\def \P{\mathcal{P}}
\def \I{\mathcal{I}}
\def \p{\mathfrak{p}}
\def \m{\mathfrak{m}}
\def \a{\mathfrak{a}}

\def \q{\mathfrak{q}}
\def \Af{\mathbb{A}_{K,f}}
\def \R{\mathbb{R}}

\def \Q{\mathbb{Q}}

\def \Z{\mathbb{Z}}
\def \N{\mathbb{N}}
\def \BB{\mathbb{B}}
\def \KK{\mathbb{K}}
\def \H{\mathcal{H}}
\def \Prim{\mathrm{Prim}}
\def \Ind{\mathrm{Ind}}

\title{Primitive ideals and K-theoretic approach to Bost-Connes systems}
\author{Takuya Takeishi}
\address{Department of Mathematical Sciences, University of Tokyo}
\email{takeishi@ms.u-tokyo.ac.jp}
\date{}

\begin{document}

\begin{abstract}
By KMS-classification theorem, the Dedekind zeta function is an invariant of Bost-Connes systems. 
In this paper, we show that it is in fact an invariant of Bost-Connes $C^*$-algebras. 
We examine second maximal primitive ideals of Bost-Connes $C^*$-algebras, and apply $K$-theory to some quotients. 
\end{abstract}

\maketitle

\section{Introduction}
For a number field $K$, there is a $C^*$-dynamical system $\mathcal{A}_K = (A_K,\sigma_t)$ the so called {\it Bost-Connes system} for $K$. 
Its extremal KMS$_{\beta}$-states for $\beta > 1$ correspond to the Galois group $\GKab = G(\Kab/K)$, 
and its partition function coincides with the Dedekind zeta function $\zeta_K(\beta)$ (cf.~\cite{LLN}). 
Since the zeta function naturally arises as an operator algebraic invariant, 
we can naturally ask whether or not we can get other number theoretic invariants from operator algebraic tools. 
This seems to be optimal because we use the class field theory to construct the $C^*$-algebra $A_K$, 
although the partition function does not reflect such a construction so much. 
So we can expect that even the $C^*$-algebra $A_K$ itself may have plenty of information of the original field $K$. 
Indeed, the narrow class number $h_K^1$ is an invariant of $A_K$ by \cite{T}, looking at maximal primitive ideals of $K$. 

In this paper, we proceed this strategy to the next step. We look at second maximal primitive ideals (Definition \ref{secmax}). 
There is a good relation between second maximal primitive ideals and prime ideals of the integer ring $\O_K$ as in Proposition \ref{summary}. 
So we can expect that the primitive ideals (or the families of them) corresponding to a fixed prime $\p$ may remember some information of $\p$. 
Indeed, the information of $\p$ can be recovered as in Theorem \ref{main2} by using $K$-theory. 
As a consequence, we obtain the following classification theorem: 
\begin{thm} \label{main1}
Let $K,L$ be number fields and $A_K,A_L$ be associated Bost-Connes $C^*$-algebras. 
If $A_K$ is isomorphic to $A_L$, then we have $\zeta_K=\zeta_L$. 
\end{thm}
As we said, $\zeta_K$ is an invariant of Bost-Connes systems. 
Theorem \ref{main1} says it is in fact an invariant of Bost-Connes $C^*$-algebras. 
Generally speaking, in order to classify a class of non-simple $C^*$-algebras, it is effective to combine ideal structure theory and $K$-theory. 
Our strategy goes in a usual way in this sense, but how to use the information of ideals seems to depend on $C^*$-algebras. 
So we think our strategy is interesting as an concrete example of classifying a class of non-simple $C^*$-algebras. 

Taking semigroup $C^*$-algebras $C^*_r(\O_K \rtimes \O_K^{\times})$ is another way to construct $C^*$-algebras related to number fields. 
For semigroup $C^*$-algebras, there is a work of Li \cite{Li1} for the classification of such $C^*$-algebras.  
According to \cite{Li1}, the minimal primitive ideals of the semigroup $C^*$-algebras $C^*_r(R\rtimes R^{\times})$ are labeled by prime ideals of $R$, 
and we can extract some information of original prime ideals by looking at $K$-theory of the quotient. 
This work is inspired by Li's work, although the proof is much different. 
It is interesting that there seems to exist a common philosophy behind two different constructions. 

Theorem \ref{main1} follows from Theorem \ref{main2} and our goal is to show Theorem \ref{main2}. 
In section \ref{defn}, we recall the construction of Bost-Connes $C^*$-algebras. 
We give a proof of Theorem \ref{main2} in Section \ref{pf}. 

\section{Definitions of Bost-Connes systems} \label{defn}
In this section, we quickly review the definition of Bost-Connes systems for number fields. 
For detail, the reader can consult \cite[p.388]{Y} or \cite[Section 2]{T}. 

Let $K$ be a number field. 
The symbol $J_K$ denotes the group of fractional ideals of $K$, and $I_K$ denotes the semigroup of integral ideals of $K$.
Define a compact space $Y_K$ by  
\begin{equation*}
Y_K = \hat{\O}_K \times_{\hat{O}_K^*} \GKab,
\end{equation*}
where $\hat{\O}_K$ is the profinite completion of $\O_K$, 
and $\hat{\O}_K^*$ acts on $\hat{\O}_K$ by multiplication and on $\GKab$ via Artin reciprocity map. 

The ideal $\a \in I_K$ acts on $Y_K$ by  
\begin{equation*}
\a \cdot [\rho,\alpha] = [\rho a, [a]_K^{-1}\alpha], 
\end{equation*}
where $a$ is a finite id\'ele which generates $\a$. 

\begin{dfn}
The $C^*$-algebra $A_K = C(Y_K)\rtimes I_K$ is called the {\it Bost-Connes $C^*$-algebra} for $K$. 
\end{dfn}

The term {\it Bost-Connes system} for $K$ usually means the $C^*$-dynamical system $\mathcal{A}_K = (A_K,\sigma_t)$, 
where the time evolution $\sigma_t$ is determined by the ideal norm. 
The time evolution adds the information of inertia degrees of primes to $A_K$, 
and such information can be recovered as the partition function (cf.~ \cite{LLN}). 
However, we do not use the time evolution and we focus on the $C^*$-algebra $A_K$ itself in this paper. 

The Bost-Connes $C^*$-algebra $A_K$ can be written as a full corner of a non-unital group crossed product. 
Let 
\begin{equation*}
X_K = \Af \times_{\hat{\O}_K^*} \GKab, 
\end{equation*}
where $\Af$ is the finite ad\'ele ring of $K$. Let $\tilde{A}_K = C_0(X_K)\rtimes J_K$. 
Then $A_K = 1_{Y_K}\tilde{A}_K1_{Y_K}$ and $1_{Y_K}$ is a full projection. 
This presentation is crucial to determine the primitive ideal space of $A_K$ in \cite{T}.

\section{Proof of the Main Theorem} \label{pf}

\subsection{Arithmetic Observations} \label{arith}
First we fix notations (basically, we follow notations of \cite{N}). Let $K$ be a number field. 
The set of all finite primes is denoted by $\P_K$. 
The symbol $K^*_+$ denotes the group of all totally positive nonzero elements of $K$ and let $\O_{K,+}^{\times} = \O_K \cap K_+^*$. 
For any ideal $\m$ of $\O_K$, Let 
\begin{eqnarray*}
&&J_K^{\m} = \{ \a \in J_K\ |\ \a \mbox{ is prime to } \m \}, \\
&&P_K^{\m} = \{ (k) \in J_K\ |\ k \in K^*_+, k \equiv 1 \mod \m \}. 
\end{eqnarray*}
Similarly, for any subset $S$ of $\P_K$, $J_K^S$ is the set of all fractional ideals which is prime to any $\p \in S$, and $I_K^S = I_K \cap J_K^S$. 
For any finite prime $\p$ of $K$, $K_{\p}$ denotes the localization of $K$ at $\p$, and $\O_{\p}$ denotes the integer ring of $K_{\p}$. 
The unit group $\O_{\p}^*$ is often denoted by $U_{\p}$. 
For any integer $m\geq 1$, let $U_{\p}^{(m)} = 1 + \p^m$ and $U_{\p}^{(0)} = U_{\p}$. 

For a ring $R$, $R^{\times}$ denotes $R \setminus \{0\}$. 

Fix a finite prime $\p$ of $K$ which is above a rational prime $p \in \Q$. Let
\begin{eqnarray*}
&&U = \{1\} \times \prod_{\q \neq \p} \O_{\q}^* \subset \Af^*, \\
&&G_{\p} = \Af^*/\overline{K^*_+}U.
\end{eqnarray*}
The group $G_{\p}$ plays an important role later. We can see that the group $J_K^{\p}$ can be considered as a subgroup of $G_{\p}$. 
The important point is that this is a dense subgroup. This follows from the following proposition: 

\begin{prop} \label{inverselimit}
We have
\begin{equation*}
G_{\p} = \lim_{\longleftarrow m} J_{K}^{\p}/P_K^{\p^m}.
\end{equation*}
\end{prop} 
\begin{proof}
First, we define a homomorphism $\varphi_m : G_{\p} \rightarrow J_K^{\p}/P_K^{\p^m}$. 
For $a \in \Af^*$, take $k\in K^*_+$ such that $a_{\p}k \equiv 1 \mod \p^m$, then define $\tilde{\varphi}_m(a)=(ak)$. 
Then it is independent of the choice of $k$, and $\tilde{\varphi}_m$ is a group homomorphism from $\Af^*$ to $J_K^{\p}/P_K^{\p^m}$. 
The homomorphism $\tilde{\varphi}_m$ is trivial on $\overline{K^*_+}U$ because $\overline{K^*_+}$ is contained in the open subgroup 
$K^*_+ U_{\p}^{(m)}U$. Hence, it induces a group homomorphism $\varphi_m : \Af^*/\overline{K^*_+}U \rightarrow J_K^{\p}/P_K^{\p^m}$. 
One can see that this homomorphism is open and continuous. 

Let us determine the kernel of $\varphi_m$. 
Clearly, $\ker \varphi_m$ contains the open subgroup $K^*_+U_{\p}^{(m)}U / \overline{K^*_+}U$. 
For the reverse inclusion, let $a \in \Af^*$ such that $\varphi_m(\bar{a})=1$. 
Here, $\bar{a}$ means the image of $a$ in $G$. 
Then we can take $k,l \in K^*_+$ such that $(ak)=(l)$ and $ak \equiv 1, l\equiv 1$ $\mod \p^m$. 
This means $akl^{-1} \in U_{p}^{(m)} U$. So we have $\ker \varphi_m = K^*_+U_{\p}^{(m)}U / \overline{K^*_+}U$.

The homomorphisms $\varphi_m$ commutes with the projective system $J_{K}^{\p}/P_K^{\p^m}$, so it induces a homomorphism 
$\displaystyle G_{\p} \rightarrow \lim_{\longleftarrow m} J_{K}^{\p}/P_K^{\p^m}$. 
It is automatically surjective. To see that it is injective, it suffices to check $\bigcap_m \ker \varphi_m = 1$. 
Take $a \in \Af^*$ such that $\bar{a} \in \bigcap_m \ker \varphi_m$, and we show that $\bar{a} = 1$. 
By multiplying an element of $\O_{K,+}^{\times}$, we may assume that $v_{\q}(a) \geq 0$ for any $\q$. 
Let
\begin{eqnarray*}
C = \{x \in \Af^*\ |\ v_{\q}(a)\geq v_{\q}(x)\geq 0 \mbox{ for any } \q \}. 
\end{eqnarray*}
Then $C$ is a compact open subset of $\Af^*$ containing $\hat{\O}^*_K$. 
By assumption, for any $m$ there exist $k_m \in K^*_+$ and $a_m \in U_{\p}^{(m)}U$ such that $a = k_m a_m$. 
Since we have $k_m = aa_m^{-1} \in C$, we can take an accumulation point $k \in \overline{K^*_+} \cap C$ and 
a subsequence $\{k_{m_j}\}$ of $\{k_m\}$ converging to $k$. 
Then the sequence $\{a_{m_j}\}$ converges to $ak^{-1}$. 
This implies $ak^{-1} \in \bigcap_m U_{\p}^{(m)} U = U$, so $a \in \overline{K^*_+}U$, which completes the proof. 
\end{proof}

Proposition \ref{inverselimit} gives an inductive limit structure of the $C^*$-algebra $C(G_{\p}) \rtimes J_K^{\p}$, 
which is useful to look into $K$-theory. 
The next lemma is used to examine the connecting maps of the inductive limit. 

\begin{lem} \label{order}
Let $f$ be the inertia degree of $\p$ over $p$. For any $m \geq 1$, We have 
$[P_K^{\p^m}: P_K^{\p^{m+1}}] = p^f$. 
\end{lem}
\begin{proof}
The group $P_K^{\p^m}/P_K^{\p^{m+1}}$ can be embedded into $U_{\p}^{(m)} / U_{\p}^{(m+1)}$. 
We show that this map is an isomorphism. 
 Let $a \in \p^m$, and take $k \in \O_{K,+}^{\times}$ such that $a\equiv k \mod \p^{m+1}$. 
Then $l=1+k$ is in $P_K^{\p^m}$, and $l(1+a)^{-1} \in U_{\p}^{(m+1)}$. This implies the surjectivity, and the injectivity is clear. 

The group 
$U_{\p}^{(m)} / U_{\p}^{(m+1)}$ is isomorphic to the additive group $\kappa_{\p}$, where $\kappa_{\p}$ is the residual field $\O_K / \p$ 
(see \cite[Chapter II, Proposition 3.10]{N}). 
The order of $\kappa_{\p}$ equals to $p^f$, so we have $[P_K^{\p^m}: P_K^{\p^{m+1}}] = p^f$. 
\end{proof}

\subsection{Primitive ideals revisited}
In \cite[Section 3.3]{T}, the primitive ideal structure of $A_K$ is formally determined. 
$\Prim A_K$ has a bundle structure over $2^{\P_K}$ with torus fibers. 
Although the concrete form of fibers has not been determined yet except the case of $\Q$ or imaginary quadratic fields, 
the formal description is useful for our purpose. 

\begin{dfn}[\cite{T}]
For a subset $S$ of $\P_K$, define a subgroup $\Gamma_S$ of $P_K^1$ by 
\begin{equation*}
\Gamma_S = \{(a)\ |\ a \in \overline{K^*_+} \subset \Af^*, a_{\p}=1 \mbox{ for } \p \not\in S \}. 
\end{equation*} 
\end{dfn}

\begin{thm}[\cite{T}] The following holds: 
\begin{enumerate}
\item $\Prim A_K$ is identified with $\displaystyle \bigcup_{S \subset \P_K} \hat{\Gamma}_S$, where $\hat{\Gamma}_S$ is the Pontrjagin dual of $\Gamma_S$. 
Let $P_{S,\gamma}$ be the ideal which corresponds to $\gamma \in \hat{\Gamma}_S$. Then we have 
\begin{equation*}
P_{S,\gamma} = \ker ((\Ind_{\Gamma_S}^{J_K} (\mathrm{ev}_x \rtimes \gamma))|_{A_K}),
\end{equation*}
where $x=[\rho,\alpha] \in X_K$ which satisfies that $\rho_{\p}= 0$ if and only if $\p \in S$. 
\item The canonical map
\begin{equation*}
2^{\P_K} \times \hat{J}_K \rightarrow \bigcup_{S \subset \P} \hat{\Gamma}_S = \mathrm{Prim} A_K,
\ (S, \gamma) \mapsto \gamma|_{\Gamma_S} \in \hat{\Gamma}_S
\end{equation*}
is an open continuous surjection, where the topology of $2^{\P_K}$ is power-cofinite topology. 
\end{enumerate}
\end{thm}

Note that the isotropy group of $x$ in (1) is $\Gamma_S$. 

\begin{rmk} \label{rmk}
The explicit form of the representation $(\Ind_{\Gamma}^{J_K} (\mathrm{ev}_x \rtimes \gamma))|_{A_K}$ can be determined. 
Let $\tilde{\H}_{S,\gamma} = \ell^2(J_K/\Gamma_{S})$. 
Take a lift of $\gamma$ on $\hat{J}_K$, which is still denoted by $\gamma$. 
Define a representation $\tilde{\pi}_{S,\gamma}$ of $\tilde{A}_K$ by 
\begin{equation*}
\tilde{\pi}_{S,\gamma}(f)\xi_{\bar{t}} = f(tx)\xi_{\bar{t}},\ \tilde{\pi}_{S,\gamma}(u_s)\xi_{\bar{t}} = \gamma(s) \xi_{\bar{st}}, 
\end{equation*}
where $s,t \in J_K$ and $f \in C_0(X_K)$. 
Up to unitary equivalence, $\pi_{S,\gamma}$ is independent of the choice of a lift of $\gamma$. 
Then 
$(\pi_{S,\gamma}, \H_{S,\gamma}) = (\tilde{\pi}_{S,\gamma}|_{A_K}, \tilde{\pi}_{S,\gamma}(1_{Y_K})\tilde{\H}_{S,\gamma})$
is exactly the above irreducible representation. 
Moreover, we have $\tilde{\pi}_{S,\gamma}(1_{Y_K})\tilde{\H}_{S,\gamma} = \ell^2(J_K^{S^c}/\Gamma_S \times I_K^S)$. 
Note that $\pi_{S,\gamma}(v_{\p})$ is a unitary if $\p \not\in S$. 

Note that the unitary equivalence class of the representation $\pi_{S,\gamma}$ may change if we change $x$ to another one 
--- only the kernel of $\pi_{S,\gamma}$ is independent of the choice of $x$. 
\end{rmk}

\begin{prop} \label{latticestr}
Let $S,T$ be two subsets of $\P_K$ and let $\gamma \in \hat{J}_K$. 
Then we have $P_{S,\gamma} \subset P_{T,\gamma}$ if and only if $S\subset T$. 
\end{prop}
\begin{proof}
Take $x=[\rho,\alpha],y=[\sigma,\beta] \in Y_K$ satisfying that $\rho_{\p}=0$ 
if and only if $\p \in S$, and $\sigma_{\p}=0$ if and only if $\p \in T$. 

Suppose $P_{S,\gamma} \subset P_{T,\gamma}$. 
Since $C_0((\overline{J_Kx})^{c} \cap Y_K) \subset P_{S,\gamma}$, any function of $C(Y_K)$ which vanishes on $\overline{J_Kx} \cap Y_K$ 
also vanishes on $\overline{J_Ky} \cap Y_K$. Hence $\overline{J_Kx} \cap Y_K \supset \overline{J_Ky} \cap Y_K$, which is equivalent to $S\subset T$ 
by \cite[Lemma 3.12]{T}. 

Suppose $S\subset T$. By \cite[Lemma 3.12]{T}, there exists a sequence $\{w_n\}_{n \in \N} \subset J_K$ such that $w_n x$ converges to $y$ in $X_K$. 
For any $n \in \N$, let $\tilde{\pi}_n$ be the representation of $\tilde{A}_K$ on $\ell^2(J_K/\Gamma_S)$ determined by 
\begin{equation*}
\tilde{\pi}_n(f)\xi_{\bar{t}} = f(tw_nx)\xi_{\bar{t}},\ \tilde{\pi}_n(u_s)\xi_{\bar{t}} = \gamma(s) \xi_{\bar{st}}, 
\end{equation*}
where $s,t \in J_K$ and $f \in C_0(X_K)$. Then we can see that $\tilde{\pi}_n$ is unitarily equivalent to $\tilde{\pi}_{S,\gamma}$. 
For any $a \in \tilde{A}_K$, we can see that there exists a limit of $\tilde{\pi}_n(a)$ with respect to the strong operator topology, 
which is denoted by $\pi(a)$. The representation $\pi$ of $\tilde{A}_K$ is determined by 
\begin{equation*}
\pi(f)\xi_{\bar{t}} = f(ty)\xi_{\bar{t}},\ \pi(u_s)\xi_{\bar{t}} = \gamma(s) \xi_{\bar{st}}, 
\end{equation*}
for $s,t \in J_K$ and $f \in C_0(X_K)$. If $a \in P_{S,\gamma}$, then we have $\pi(a)=0$ by definition. 
Hence it suffices to show that $\ker \pi \cap A_K$ is contained in $P_{T,\gamma}$. 

Let $\rho = \tilde{\pi}_{T,\gamma}$. First, we consider the case of $\gamma = 1$. 
By the universality of group crossed products, we have a quotient map $\phi:C_0(X_K) \rtimes J_K \rightarrow C_0(\overline{J_Kx}) \rtimes (J_K/\Gamma_S)$. 
Then the representations $\pi,\rho$ factors through $\phi$, i.e., there exist representations $\pi', \rho'$ of $C_0(\overline{J_Kx}) \rtimes (J_K/\Gamma_S)$ 
on $\BB(\ell^2(J_K/\Gamma_S))$ such that $\pi = \pi' \circ \phi$ and $\rho = \rho' \circ \phi$. 
We can see that $\pi'$ is in fact a faithful representation (cf.~\cite[Lemma 2.5.1]{CEL}). Hence $\ker \pi = \ker \phi \subset \ker \rho$. 

Next, we consider general cases. In fact, the $C^*$-algebras $\pi(\tilde{A}_K), \rho(\tilde{A}_K)$ are independent of the choice of $\gamma$. 
Hence we have a quotient map 
$\psi: \pi(\tilde{A}_K)\rightarrow \rho(\tilde{A}_K)$ obtained from the case of $\gamma = 1$. 
We can directly check that $\psi \circ \pi = \rho$ for general $\gamma$. Therefore, $\ker \pi \subset \ker \rho$ holds for general $\gamma$. 

\end{proof}

The proof of the following lemma is essentially the same as \cite[Proposition 3.8]{T}. 
\begin{lem}
Let $S$ be a subset of $\P_K$, and let
\begin{eqnarray*}
&&Y_K^S = \{ x=[\rho,\alpha] \in Y_K\ |\ \rho_{\p}=0 \mbox{ for any } \p \in S \}, \\
&&P_S = C_0((Y_K^S)^c) \rtimes I_K.
\end{eqnarray*}
Then we have 
\begin{equation*}
\bigcap_{\gamma \in \hat{\Gamma}_S} P_{S,\gamma} = P_S.
\end{equation*}
In particular, $P_S$ is a primitive ideal if and only if $\Gamma_S=1$.  
\end{lem}
\begin{proof}
Here, we consider that representations $\pi_{S,\gamma}$ in Remark \ref{rmk} are defined for any $\gamma \in \hat{J}_K$ 
(some of them are mutually unitarily equivalent, but we distinguish them). 

By \cite[Lemma 3.12]{T}, $P_S$ is contained in $\ker \pi_{S,\gamma} = P_{S,\gamma}$ for any $\gamma$. 
Let $B= \BB (\pi_{S,\gamma}(1_{Y_K})\ell^2(J_K/\Gamma_{S}))$. 
Then the image of the homomorphism 
\begin{equation*}
\prod_{\gamma \in \hat{J}_K} \pi_{S,\gamma} : A_K/P_S \rightarrow \prod_{\gamma \in \hat{J}_K} B. 
\end{equation*}
is actually contained in $C(\hat{J}_K,B)$. Let $\Phi:A_K/P_S \rightarrow C(\hat{J}_K,B)$ be the restriction of that map. 
Since $\ker \Phi = \bigcap_{\gamma} \ker \pi_{S,\gamma} / P_S$, it suffices to show that $\Phi$ is injective. 

We have  
$A_K/P_S \cong C(Y_K^S) \rtimes (J_K^{S^c} \times I_K^S)$, 
and 
$\Phi(fv_{\a}) = \chi_{\a} \otimes \pi_{S,1}(fv_{\a})$ for $f \in C(Y_K^S), \a \in J_K^{S^c} \times I_K^S$, 
where $\chi_{\a}$ is the character on $\hat{J}_K$ corresponding to $\a$. 
So we have the following commutative diagram: 
\begin{eqnarray*}
\xymatrix{
C(Y_K^S) \rtimes (J_K^{S^c} \times I_K^S) \ar[r]^{\Phi} \ar[d]^{E} & C(\hat{J}_K) \otimes B \ar[d]^{\mu \otimes \mathrm{id}_B} \\
C(Y_K^S) \ar[r] & B,
} 
\end{eqnarray*}
where $\mu$ is the Haar measure of $\hat{J}_K$. The homomorphism of the bottom line is injective by \cite[Lemma 3.12]{T}. 
This implies that $\Phi$ is injective. 
\end{proof}

Next, we focus on second maximal primitive ideals. 
\begin{dfn} \label{secmax}
Let $A$ be a $C^*$-algebra and let $P$ be a primitive ideal of $A$. 
We say that $P$ is second maximal if the following holds: 
\begin{enumerate}
\item There exists a primitive ideal $Q$ of $A$ such that $P \subsetneq Q$. 
\item There does not exist a pair of primitive ideals $Q_1,Q_2$ of $A$ such that $P \subsetneq Q_1 \subsetneq Q_2$. 
\end{enumerate}
\end{dfn}

Note that a maximal primitive ideal $Q$ in the condition (1) may not be unique. 
For Bost-Connes $C^*$-algebras, second maximal primitive ideals are exactly of the form of $P_{\{\p\}^c,\gamma}$ 
for some prime $\p$ and $\gamma \in \hat{\Gamma}_{\{\p\}^c}$ by Proposition \ref{latticestr}. 
In the case of $K= \Q$ or imaginary quadratic fields, $P_{\{\p\}^c}$ is a second maximal primitive ideal. 

\begin{dfn}
Let $\I_{1,K}$ be the set of all maximal primitive ideals of $A_K$, and 
let $\I_{2,K}$ be the set of all second maximal primitive ideals of $A_K$. 
We consider $\I_{1,K}$ and $\I_{2,K}$ as topological spaces with the relative topology of $\Prim A_K$. 
\end{dfn}

\begin{lem}
The space $\I_{2,K}$ is equal to the direct sum of $\hat{\Gamma}_{\{\p\}^c}$ for all $\p$ as a topological space. 
In particular, $\I_{2,K}$ is Hausdorff. 
\end{lem}
\begin{proof}
Let 
$Q_2 = \{ \{\p\}^c \in 2^{\P_K}\ |\ \p \mbox{ is a prime} \}$. 
Then we can check that $Q_2$ is Hausdorff with respect to the relative topology of the power-cofinite topology. 
Let $\pi:2^{\P_K} \times \hat{J}_K \rightarrow \Prim A_K$ be the canonical map. Then we can see that 
the restriction $\pi:Q_2 \times \hat{J}_K \rightarrow \I_{2,K}$ is an open continuous surjection. 
This means that each $\hat{\Gamma}_{\{\p\}^c}$ is compact open inside $\I_{2,K}$. 
\end{proof}

In summary, we have the following proposition:
\begin{prop} \label{summary}
There is one-to-one correspondence between $\P_K$ and connected components of $\I_{2,K}$. 
The connected component $\mathcal{C}_{\p}$ corresponding to $\p$ is equal to $\hat{\Gamma}_{\{\p\}^c}$, and we have $\bigcap \mathcal{C}_{\p}=P_{\{\p\}^c}$. 
\end{prop}

We focus on the $C^*$-algebra $P_{\P_K}/P_{\{\p\}^c}$. 
By definition, we have 
\begin{equation*}
P_{\P_K} = \ker (C(Y_K) \rtimes I_K \rightarrow C(Y_K^{\P_K}) \rtimes J_K), 
\end{equation*}
and by \cite[Proposition 3.8]{T}, $P_{\P_K} = \bigcap \I_{1,K}$. 

\begin{lem} \label{simple}
We have 
\begin{equation*}
P_{\P_K}/P_{\{\p\}^c} \cong \KK \otimes C(G_{\p}) \rtimes J_K^{\p}, 
\end{equation*}
where $G_{\p}$ is the profinite group in Section \ref{arith}.
\end{lem}
\begin{proof}
We have 
\begin{eqnarray*}
&P_{\P_K}/P_{\{\p\}^c} &= \ker (C(Y_K^S) \rtimes (J_K^{\p} \times \p^{\N}) \rightarrow C(Y_K^{\P_K}) \rtimes J_K) \\
&&= C_0(\O_{\p}^{\times} \times_{\hat{\O}_K^*} \GKab ) \rtimes (J_K^{\p} \times \p^{\N}), 
\end{eqnarray*}
and we focus on the dynamical system $J_K^{\p} \times \p^{\N} \curvearrowright \O_{\p}^{\times} \times_{\hat{\O}_K^*} \GKab$. 
First, $\O_{\p}^{\times} \times_{\hat{\O}_K^*} \GKab$ is naturally identified with 
\begin{equation*}
\O_{\p}^{\times} \times_{\O_{\p}^*} (\GKab/U) = \O_{\p}^{\times} \times_{\O_{\p}^*} G_{\p}, 
\end{equation*}
where $G_{\p}$ and $U$ are as in Section \ref{arith}. 
Fix a prime element $\pi_{\p}$ of $K_{\p}$. Then $\displaystyle \O_{\p}^{\times} \times_{\O_{\p}^*} G_{\p}$ is homeomorphic to $\N \times G_{\p}$ 
by sending $[\pi_{\p}^n, \alpha]$ to $(n,\alpha)$ for $n \in \N$ and $\alpha \in G_{\p}$. 
Under this identification, the action is identified with the following action: 
\begin{equation*}
\q (n,\alpha) = (n, \q^{-1}\alpha),\ \p (n,\alpha) = (n+1, \pi_{\p}^{-1}\alpha). 
\end{equation*}

Moreover, by the homeomorphism of $\N \times G_{\p} \rightarrow \N \times G_{\p}$ defined by $(n, \alpha)\mapsto (n, \pi_{\p}^n\alpha)$, 
this action is conjugate to the following action: 
\begin{equation*}
\q (n,\alpha) = (n, \q^{-1}\alpha),\ \p (n,\alpha) = (n+1, \alpha). 
\end{equation*}

Hence, we have 
\begin{equation*}
P_{\P_K}/P_{\{\p\}^c} \cong C_0(\N \times G_{\p}) \rtimes_{\alpha \otimes \beta} \N \times J_K^{\p} 
\cong (C_0(\N) \rtimes_{\alpha} \N) \otimes (C(G_{\p}) \rtimes_{\beta} J_K^{\p}), 
\end{equation*}
where $\alpha$ is the action by addition and $\beta$ is the action by multiplication ($J_K^{\p}$ is naturally identified with a subgroup of $G_{\p}$). 
The second isomorphism is established by writing both algebras as corners of group crossed products and applying 
the well-known decomposition theorem for tensor product actions.  
$C_0(\N) \rtimes_{\alpha} \N$ is isomorphic to $\KK$, because it is written as a full corner of $C_0(\Z)\rtimes \Z$. 
\end{proof}

\begin{rmk}
Brownlowe-Larsen-Putnam-Raeburn give a similar presentation in the case of $K=\Q$. 
Lemma \ref{simple} is a kind of generalization of \cite[Theorem 4.1 (2)]{BLPR} in the case of $\P \setminus S$ is one point. 
\end{rmk}

\subsection{$K_0$-groups of profinite actions}
In this section, we prepare a general machinery which is used in the next section. 
Let $\Gamma$ be a discrete amenable torsion-free group. 
In our case, we only need the case of $\Gamma = \Z^{\infty}$, but here we treat the general case. 
Let $P_m$ be a decreasing sequence of finite index normal subgroups of $\Gamma$ such that $\bigcap_m P_m = 1$. 
Define a profinite group $G$ by $\displaystyle G= \lim_{\longleftarrow m} \Gamma /P_m$. 
Then, by assumption, $\Gamma$ is a dense subgroup of $G$ and $G/\overline{P_m} = \Gamma /P_m$. 
Let $h_m = [\Gamma : P_m]$. 
The $C^*$-algebra $C(G)\rtimes \Gamma$ is simple and has a unique tracial state. 

\begin{prop} \label{trace}
Let $\tau$ be the unique tracial state of $C(G)\rtimes \Gamma$. Then we have 
\begin{equation*}
\tau_* (K_0(C(G)\rtimes \Gamma)) = \bigcup_m h_m^{-1}\Z. 
\end{equation*}
\end{prop}
\begin{proof}
By assumption, we have
\begin{equation*}
C(G)\rtimes \Gamma \cong \lim_{\longrightarrow } C(\Gamma/P_m) \rtimes \Gamma, 
\end{equation*}
and by the imprimitivity theorem, $C(\Gamma/P_m) \rtimes \Gamma$ is Morita equivalent to $C^*_r(P_m)$. 
In this case, $1_{P_m}C^*_r(P_m) = 1_{P_m}(C(\Gamma/P_m) \rtimes \Gamma)1_{P_m}$ and $1_{P_m}$ is a full projection of $C(\Gamma/P_m) \rtimes \Gamma$. 
The inclusion $1_{P_m}C^*_r(P_m) \hookrightarrow C(\Gamma/P_m) \rtimes \Gamma$ gives the isomorphism of $K$-groups. 

Let $\tau$ be the unique tracial state of $C(G)\rtimes \Gamma$. Then it is equal to $\mu \circ E$, where $\mu$ is the Haar measure of $G$ and 
$E$ is the canonical conditional expectation $C(G)\rtimes \Gamma \rightarrow C(G)$. 
Then the restriction of $\tau$ onto $C(\Gamma/P_m) \rtimes \Gamma$ is equal to $\mu_m \circ E$, where $\mu_m$ is 
the normalized counting measure of $\Gamma/P_m$. 
Hence the restriction of $\tau$ onto $1_{P_m}C^*_r(P_m)$ is equal to $h_m^{-1}\tau^{P_m}$, where $\tau^{P_m}$ is the canonical trace of $C^*_r(P_m)$. 

Since $P_m$ is torsion-free and satisfies the Baum-Connes conjecture (cf. \cite[Proposition 6.3.1]{V}), we have
\begin{equation*}
\tau^{P_m}_*(K_0(C^*_r(P_m))) = \Z 
\end{equation*}
as a subgroup of $\R$. 
Hence 
\begin{eqnarray*}
&\tau_*(K_0(C(G)\rtimes \Gamma)) &= \bigcup_m \tau_* (C(\Gamma/P_m) \rtimes \Gamma) \\
&&= \bigcup_m h_m^{-1}\tau_*^{P_m}(C^*_r(P_m)) \\
&&= \bigcup_m h_m^{-1}\Z. 
\end{eqnarray*}

\end{proof}

In our case we actually have to treat with unbounded traces. In this paper, 
we tacitly assume unbounded traces have finite values on finite projections. 
The following lemma is proved in usual way: 
\begin{lem} \label{unique}
Let $A$ be a simple $C^*$-algebra with a unique tracial state $\tau$. 
Then $\KK \otimes A$ has a unique unbounded trace $\mathrm{Tr} \otimes \tau$ up to scalar multiplication. 
\end{lem}

\subsection{Conclusions}

\begin{thm} \label{main2}
Let $K$ be a number field and let $\p$ be a finite prime of $K$, and let $\p \cap \Z = (p)$. 
Then $\bigcap \I_{1,K}/ \bigcap \mathcal{C}_{\p}$ is a simple $C^*$-algebra with a unique unbounded trace $T$ up to scalar multiplication, and we have 
\begin{equation*}
T_*(K_0 (\bigcap \I_{1,K}/\bigcap \mathcal{C}_{\p})) \cong \Z[1/p]. 
\end{equation*}
\end{thm}
\begin{proof}
By Proposition \ref{summary} and Lemma \ref{simple}, the $C^*$-algebra $\bigcap \I_{1,K}/ \bigcap \mathcal{C}_{\p}$ 
is isomorphic to $\KK \otimes C(G_{\p}) \rtimes J_K^{\p}$. 
Let $\tau$ be the unique trace of $C(G_{\p}) \rtimes J_K^{\p}$. 
By Lemma \ref{unique}, we may assume $T= \mathrm{Tr} \otimes \tau$ because the isomorphism class of $T_*(K_0 (\bigcap \I_{1,K}/\bigcap \mathcal{C}_{\p}))$ is 
independent of the choice of $T$. 
By Proposition \ref{inverselimit}, we have 
\begin{equation*}
C(G_{\p}) \rtimes J_K^{\p} \cong \lim_{\longrightarrow m} C(J_K^{\p}/P_K^{\p^m}) \rtimes J_K^{\p}. 
\end{equation*}
So by applying Proposition \ref{trace}, we have 
\begin{equation*}
T_*(K_0(\bigcap \I_{1,K}/ \bigcap \mathcal{C}_{\p})) = \tau_*(K_0(C(G_{\p}) \rtimes J_K^{\p})) = \bigcup_{m} h_m^{-1} \Z, 
\end{equation*}
where $h_m=[J_K^{\p}:P_K^{\p^m}]$. By Lemma \ref{order}, we have $h_{m+1}/h_m = p^f$, so $\bigcup_{m} h_m^{-1} \Z \cong \Z[1/p]$. 
\end{proof}

The proof of Theorem \ref{main1} is obtained by applying the theorem of Stuart-Perlis \cite{SP2}. 
For a rational prime $p$, $g_K(p)$ denotes the splitting number of $p$, i.e., the number of primes of $K$ which is above $p$. 

\begin{proof}[Proof of Theorem \ref{main1}]
By Theorem \ref{main2} and Proposition \ref{summary}, $g_K(p)$ is equal to 
the number of connected components $\mathcal{C}$ of $\I_{2,K}$ which satisfy $T_*(K_0(\bigcap \I_{1,K}/\bigcap \mathcal{C})) = \Z[1/p]$ 
for some unbounded trace $T$ of $\bigcap \I_{1,K}/\bigcap \mathcal{C}$. 
Since $\Z[1/p]\cong \Z[1/q]$ if and only if $p=q$ for any rational prime $p,q$, this number is preserved under the isomorphism. 
By \cite[Main Theorem]{SP2}, the equality of splitting numbers for all rational primes implies the equality of zeta functions. 
\end{proof}

Zeta functions consist of the information of the rational prime which is below a prime and the inertia degree. 
Since we applied a number theoretic theorem, the inertia degree is not naturally obtained. 
It may be interesting to ask how to get the inertia degree in an operator algebraic way from the $C^*$-algebra $A_K$. 

\section*{Acknowledgments}
This work was supported by the Program for Leading Graduate Schools, MEXT, Japan and JSPS KAKENHI Grant Number 13J01197. 
The author would like to thank to Yasuyuki Kawahigashi, Joachim Cuntz and Xin Li for fruitful conversations. 

\bibliographystyle{amsplain} 

\bibliography{ref}

\end{document}